\documentclass{amsart}
\usepackage{amssymb,amsmath}
\usepackage{hyperref}
 \newtheorem{theorem}{Theorem}[section]
 \newtheorem{corollary}{Corollary}[section]
 \newtheorem{lemma}{Lemma}[section]
\begin{document}
\title{Commutator formulas for gradient Ricci  shrinker and their application to linear stability}
\author{MANSOUR MEHRMOHAMADI}
\address{	
	Department of Mathematics and Computer Science\\
	Mahani Mathematical Research Center
	Shahid Bahonar University of Kerman
	Kerman, Iran}
\email{mansoor\_mehrmohamadi@math.uk.ac.ir}
		
\author{ASADOLLAH RAZAVI}
 \address{	
    	Department of Mathematics and Computer Science
    	Mahani Mathematical Research Center
    	Shahid Bahonar University of Kerman
    	Kerman, Iran}
 \email{arazavi.uk.ac.ir}
\date{\today}
\thanks{At this moment we have not yet decided wich journal to publish this article in.}


\begin{abstract}
In this paper we have found some commutator formulas between $div_{f},\Delta_{f,L},div_{f}^\dag,\Delta_{f}$ on closed orientable $GRS^+$ metrics(Gradient Ricci Shrinker), then with  them we have generalized a Theorem   of Cao and Zhu about necessary condition for linear stability of a $GRS^+$.
\end{abstract}
\maketitle
\section{Introduction}
\numberwithin{equation}{section}
A complete Riemannian manifold is called  \text{\it{ Ricci Soliton}}  if there exist a complete vector field $X$ and a real constant $\lambda$ such that
$$\mathit {Ric + { \mathcal{L}_{X}}g = \lambda g}$$
where Ric is Ricci tensor, 
a soliton is called  \text{\it{expanding,steady or shrinking}} if $\lambda<0,\lambda=0\ or\lambda>0$ respectively, for degenerate case $X=0$ soliton is called  \text{\it{ trivial}}, therefore Einstien manifolds are  special cases of Ricci solitons.  Whenever   $X=\nabla f$  for some smooth function $f \in {C^\infty }(M)$  the soliton is called  \text{\it {gradient}}, in this case for shrinking soliton we have 
\begin{equation}\label{sol}
	Ric + {\nabla ^2}f = \frac{1}{{2\tau }}g
\end{equation}
where $\tau$ is a positive constant, 
  $f$ is called  the  \text{\it{  potential function }}of soliton. In this paper we denote gradient Shrinking soliton metrics briefly with $GRS^{+}$.  For more information on Ricci soliton see \cite{cao}.

The concept of Ricci soliton was invented by Richard Hamilton in mid 80's as Riemannian metrics that up to diffeomorphism and scaling are fixed points of the Ricci flow equation.
Perelman in a remarkabale paper \cite{per} discovered severl variational structures for Ricci flow, one of them is the $\nu$-entropy. He proved that critical points of the  $\nu$-entropy are $GRS^+$ metrics, therefore this question arose that is this true that $GRS^+$ metrics are the local maximum of $\nu$-entropy, trivialy if the second variation of the $\nu$-entropy is positive, then soliton can not be the local maximum of  $\nu$-entropy but suppose that the second variation of the $\nu$-entropy alway is  nonpositive,  then is this true that soliton is the local maximum of  $\nu$-entropy ? Therefore it was necessary to calculate the 
second variation of the $\nu$-entropy.  Hamilton,Ilmanen and Cao in \cite{hail} calculated the second variation of the $\nu$-entropy for positive Einstein metrics. They  defined a $GRS^+$ which is \text{\it{linear stable}}   whenever the second variation of  the $\nu$-entropy  is nonpositive and otherwise  \text{\it{linear unstable}}.
   Hamilton conjectured that at least in dimension four, only linear stable $GRS^+$ are Einstien metrics with positive scalar curvature.
  Cao and Zhu in \cite{cao zhu}  calculated the second variation of the $\nu$-entropy for nontrivial $GRS^+$, they proved that for linear stability of $GRS^+$ it is neccesary that the first eigenvalue of the weighted  Lichnerowicz Laplace operator $\Delta_{f,L}$ restricted to transversal tensor (i.e $div_{f}h=0$) is not greater than zero and only eigentensor of zero shoud be Ricci tensor  $Ric$.

Finally Kr\"oncke in \cite{kroncke} proved that if every infinitesimal solitonic deformation of soliton is integrbale,  then $GRS^+$ is the local maximum of $\nu$-entropy if and only if the second variation of the $\nu$-entropy is nonpositive. 

 Kr\"oncke proved that this condition  is failed for  complex projective space  $\mathbb{C}{P^n}$ with Fubiny-Study metric  and although complex projective space  $\mathbb{C}{P^n}$ with Fubiny-Study metric is a linear stable $GRS^+$, but this $GRS^+$  is not the local maximum of $\nu$-entropy in the space of all of Riemannian metrics on $\mathbb{C}{P^n}$.\\

Subject of this paper is closed orientable gradient Ricci shrinking soliton (briefly $GRS^+$).  In the second section we will introduce our conventions
and necessary definitions for our work, then we bring the necessary formulas and theorems without proofs from other papers.
In the third section we will get several commutator formulas between some differential operators. We prove that
\begin{theorem}
For a closed  orientable $GRS^+$ $(M^{n}, g,f,\tau)$  and  $h \in {C^\infty }({S^2}({T^*}M)),\omega\in {\Omega^{1}(M)} $ and $a\in C^{\infty}(M)$ we have
	\begin{eqnarray}
		{\Delta _f}da &=& d{\Delta _f}a +\frac{1}{{2\tau }}da \nonumber\\
		di{v_f}{\Delta _f}\omega &=&  {\Delta _f}di{v_f}\omega  + \frac{1}{{2\tau }}di{v_f}\omega\nonumber\\
		{\mathcal{L}_{\# ({\Delta _f}\omega )}}g&=&	{\Delta _{f,L}}({\mathcal{L}_{\# \omega }}g)+ \frac{1}{2\tau }{\mathcal{L}_{\# \omega }}g   \nonumber
	\end{eqnarray}
\end{theorem}Note that these formulas can be extended to other kinds of the Ricci soliton and noncompact Ricci soliton, but this is not the subject of this paper. These commutator formulas are used to investigate relations between  Ricci  curvature bounds, spectrum of  the weighted Lichnerowicz  Laplacian  and properties of eigentensors of the weighted Lichnerowicz Laplacian.
In the fourth section inspired by work of Cao and He\cite{sym}, we will prove that the stability operator  $N$ of a $GRS^+$ on  $Im (div_f^\dag )$ equals to zero pointwisely.

\begin{theorem}
	 For a closed  orientable $GRS^+$ $(M^{n}, g,f,\tau)$, the stability operator(Jacobi field operator) of  the $\nu$-entropy $N$ on  $Im (div_f^\dag )$   pointwisely equals to zero, in other word for every vector field $X\in {\chi(M)}$ $$N({\mathcal{L}_{X }}g)=0$$
	\end{theorem}
In fifth section we will extend Theorem (1.3) of Cao and Zhu in \cite{cao zhu}, we find a weaker bound other than zero $(-\frac{1}{2\tau})$.  We will prove two Theorems

\begin{theorem}
	A necessary condition for linear stability	 of closed oraiantable $GRS^+$  is that the  first eigenvalue of the weighted Lichnerowicz Laplacian $\Delta_{f,L}$ (except zero with one multiplicity and Ricci tensor  $Ric$ as eigentensor) is not greather than $-\frac{1}{2\tau}$.
\end{theorem}

\begin{theorem}
	Suppose that for a closed oraiantable $GRS^+$ the first eigenvalue of the weighted Lichnerowicz Laplacian $\Delta_{f,L}$ is not greather than $-\frac{1}{\tau}$ (except zero with one multiplicity and Ricci tensor  $Ric$ as eigentensor),  then soliton is linear stable.
\end{theorem}

 \section{Preliminaries}

\subsection{Conventions and Notations} 
 
We define Riemannian curvature as 
\begin{eqnarray}
Rm(X,Y,Z,W) &=&  -  < R(X,Y)Z,W > 
\end{eqnarray}

with this convention we have these commutator formulas 
\begin{eqnarray}
{\nabla _i}{\nabla _j}{\omega _k} &=& {\nabla _j}{\nabla _i}{\omega _k} - {R_{ijk}}^p{\omega _p}= {\nabla _j}{\nabla _i}{\omega _k} + {R_{ijkp}}{\omega ^p}\nonumber\\
{\nabla _i}{\nabla _j}{T_{pq}} &=& {\nabla _j}{\nabla _i}{T_{pq}} - {R_{ijp}}^m{T_{mq}} - {R_{ijq}}^m{T_{pm}}\nonumber \\
&=& {\nabla _j}{\nabla _i}{T_{pq}} + {R_{ijpm}}T_q^m + {R_{ijqm}}T_p^m\nonumber
\end{eqnarray}
As explained in introduction,  we need the second variation of the $\nu$-entropy.  Here we introduce some convenions (our notations is similar to \cite{cao zhu} and \cite{hail} ).

For any symmetric covariant  two tensor $h,(h_{ij})$  and any 1-form $\omega,(\omega_i)$ we denote

\begin{eqnarray}
div   \omega = &g^{pq}\nabla_{p}{\omega_q} &, (div h)_{i}=g^{pq}\nabla_{p}h_{qi} \nonumber\\ 
di{v_f}\omega  = & {e^f}div({e^{ - f}}\omega)  &, div_{f}(\omega )=div(\omega)-\omega(\nabla f)\nonumber\\
di{v_f} h   = &  {e^f}div({e^{ - f}}h) &div_{f}h= div(h ) - h(\nabla f,-)\\
div_{f}\omega=&g^{pq}(\nabla_{p}\omega_q-\omega_p\nabla_q f)&,(div_{f}h)_{i}=g^{pq}(\nabla_{p}h_{qi} -h_{pi}\nabla_q f)\nonumber \\
{\Delta _f}  = & di{v_f}\nabla =&\Delta  -\nabla_{\nabla f}\nonumber \\ (Rm(h,-))_{ij}  = &R_{piqj}h^{pq}&\nonumber\\
di{v_f}^\dag \omega =&  - \frac{1}{2}{\mathcal{L}_{\# \omega }}g,&(div^\dag_{f}\omega)_{ij}=- \frac{1}{2}(\nabla_i\omega_j+\nabla_j \omega_i)\nonumber
\end{eqnarray}
In this paper we use the weighted  $L^2$-inner product with respect to measure $dm$: $$<-,->_{dm}=\int_{M}<-,->{(4\pi \tau )^{ - \frac{n}{2}}}{e^{ - f}}dv,dm={(4\pi \tau )^{ - \frac{n}{2}}}{e^{ - f}}dv, \int_{M} dm=1$$
where $<,>$ is inner product  induced from Riemannian metric of $M$ on  arbitrary tensor bundle, $dv$ is Riemannian volume form and $f$ is the potential function of soliton. Therefore  it is important that in our notation $<,>$ is pointwise inner product on tensors and $<,>_{dm}$ is global inner product for tensor fields and $|T|$ is norm of  some tensor $T$ in some point and
 $|T|_{dm}$  is norm of tensor field $T$. Note that Cauchy–Schwarz inequality is valid for $<,>_{dm} \ and \ |-|_{dm}$. \\ 
  With this $L^2$-inner product we have the weighted divergence theorem for arbitrary 1-form $\omega\in\Omega^{1}(M)$ and smooth function $a\in C^{\infty}(M)$ and $dm={(4\pi \tau )^{ - \frac{n}{2}}}{e^{ - f}}dv$,  we have
$$\int_{M} div_{f}   (\omega )dm=0, \int_{M} \Delta_{f}adm=0$$
Here  we need to find  formal-adjoint of the defined operators with respect to the weighted $L^2$-inner product.
We have $$div_{f} (a\omega)=adiv_{f}\omega+<\omega,da>$$ therefore formal-adjoint of $div_f$ on 1-forms is $-d$. On the other hand we have  $$div_{f}(h(\omega,-))=<div_{f}h,\omega>+<h,div_{f}^{\dag}\omega>$$  therefore formal-adjoint of $div_f$ on symmetric covariant two tensors is $div_{f}^\dag$, formal-adjoint of $\nabla$ on 1-forms and symmetric covariant two tensors is $-div_{f}$  and  $\Delta_{f}$ is a self-adjoint operator. Therefore  if we denote the formal adjoint with respect to the weighted $L^2$inner product of a diffrential operator $D$ with $D^{\dag}$,  then we have $$d^{\dagger}=-div_{f},\nabla^{\dagger}=-div_{f},(div_{f})^{\dagger}=div^\dag_{f}$$ i.e

\begin{eqnarray}
\int_{M} <div_{f}\omega ,a> dm &=&\int_{M} -<\omega,da> dm\nonumber\\
\int_{M} <div_{f}h,\omega>dm&=&\int_{M} <h,div_{f}^\dag \omega,>dm\nonumber\\
\int_{M} <\Delta_{f}h,k>dm &=& \int_{M} <h,\Delta_{f}k>dm	\nonumber
\end{eqnarray}

Considering  Bakry-Emery-Ricci tensor and Ricci solitons, Lott in \cite{lott},  defined
 a weighted version of Lichnerowicz Laplacian which we work with that
$$\Delta_{f,L} h=\Delta_{f} h +2Rm(h,-)-(Ric+\nabla^2 f).h-h.(Ric+\nabla^2 f)$$ such that $Rm(h,-)_{ij}  = R_{piqj}h^{pq}$,  $(A.B)_{ij}=g^{pq}A_{ip}B_{qj}$. For special case of Ricci solitons we have 
$$\Delta_{f,L} h=\Delta_{f} h +2Rm(h,-)-\frac{1}{\tau} h$$
Trivialy $\Delta_{f,L}$ is a self-adjoint operator with respect to the measure $dm$, very importnt note is that
$$tr\Delta_{f,L}h \ne \Delta_{f}tr h$$.

\subsection{W-entropy}
For a closed orientable Riemannian manifold  $(M^n,g)$ we define Perelman's  W-entropy as 

\begin{eqnarray}
W(g,f,\tau ) & = & \int_{M} {[\tau (R + {{\left| {\nabla f} \right|}^2}) + f - n]dm} \\
dm & = & {(4\pi \tau )^{ - \frac{n}{2}}}{e^{ - f}}dv\nonumber
\end{eqnarray}
such that g is Riemannian metric , $f$ is a smooth function on $M$,$\tau$ is a positive real number,  $R$ is scalar curvature and $dv$ is volume form of manifold, wih these notations for any positive real number $c$ and diffeomorphism  $\varphi\in Diff(M)$ we have
\begin{eqnarray}
W(cg,f,c\tau ) & = & W(g,f,\tau )\nonumber\\
\label{mul}
W(g,f,\tau ) & = & W({\varphi ^*}g,{\varphi ^*}f,\tau ),\nonumber
\label{diff}
\end{eqnarray}
where ${\varphi ^*}f(p) = f(\varphi (p))$.

Then we define Perelman's $ \nu$-entropy as 
\begin{eqnarray}
& &\nu (g) =\inf\big\{ { W(g,f,\tau ):f \in {C^\infty }(M),\tau  > 0}\big\}\nonumber \\
&&\int_{M} {{{(4\pi \tau )}^{ - \frac{n}{2}}}{e^{ - f}}dv = 1}\nonumber 
\end{eqnarray}
 
 Note that  $\nu$-entropy may be infinite  and therefore   minimizing pair $(f,\tau)$ does not exist. Anyway it is proved that (\cite{cao zhu} p.5 ), if $\nu$-entropy exists, then the minimizing pair  $(f,\tau)$ should satisfies these conditions

\begin{eqnarray}\label{nce}
&&\tau ( - 2\Delta f + {{\left| {\nabla f} \right|}^2} - R) - f + n + \nu  =0\\
&&\int_{M} { f dm}={\frac{n}{2}+\nu}
\end{eqnarray}

It is proved that for  a closed  $GRS^+$,  $\nu$-entropy is finite(\cite{kroncke} Remark (3.4)) and in the $C^2$-neighborhood of $GRS^+$  in the space of all  Riemannian metrics on $M$,  $\nu$-entropy exists and is finite and there exists an unique minimizing pair $(f,\tau)$ (\cite{kroncke} Remark 3.2 and \cite{cao zhu})

\begin{theorem}[{\bf  Fist Variation Formula}]\label{fv}
The first variation of the $\nu$-entropy for a closed Riemannian manifold $(M^{n},g)$ in the pertubation direction $h \in {C^\infty }({S^2}({T^*}M))$ is given by
\begin{eqnarray}
{\nu _g}^{'}(h) = \displaystyle \int_{M}{-\tau< h,Ric + {\nabla ^2}f - \frac{1}{{2\tau }}g > dm }       
                	\end{eqnarray}
                 \end{theorem}     
                   
              \begin{proof}

              	See \cite{cao zhu} Lemma (2.2) 
                \end{proof}

From this theorem  we conclude that if minimizing pair $(f,\tau)$ satisfies solitonic equation (\ref{sol}),  then $g$ is a critical point for  $\nu$-entropy in the space of all  Riemnnian metrics on $M$.  Conversely suppose that  for a closed Riemannian manifold   $(M^n,g)$, a smooth function $a\in C^{\infty}(M)$ and positive real number $c$ we have $Ric + {\nabla ^2}a = cg$, then from Theorem(\ref{fv}) and because  $\nu$-entropy is invriant with respect to the action of diffeomorphism group and scaling of metric  it follows that ${\nu _g}^{'}({\frac{1}{2}}L_{\nabla a}g-cg)=0$  therefore   ${\nu _g}^{'}(\nabla^2a-cg)=0$
  and $$0={\nu _g}^{'}(0) ={\nu _g}^{'}(Ric+\nabla^2 a -cg)={\nu _g}^{'}(Ric)={\nu _g}^{'}(Ric+\nabla^2 f - \frac{1}{{2\tau }}g) $$
  and we conclude that $$\int_{M}|Ric+\nabla^2 f - \frac{1}{{2\tau }}g|^2 dm =0$$ and finally 
  $$Ric+\nabla^2 f = \frac{1}{{2\tau }}g$$ therefore $g$ is  critical point for $\nu$-entropy in the space of all of Riemannian metrics on $M$ if and only if minimizing pair $(f,\tau)$ satisfies solitonic equation(\ref{sol}).    On the other hand Theorem (\ref{fv}) says that under Ricci flow as well as $\nu$-entropy is  finite, $\nu$-entropy is monotone increasing  and is constant if and only if initial metric is a $GRS^+$ and minimizing pair  $(f,\tau)$ satisfies solitonic equation. Here we can state Theorem  of Cao and Zhu in \cite{cao zhu} wich states  the exact experssion of the second variation of the $\nu$-entropy.
  \begin{theorem}[{\bf  Second Variation Formula}]\label{sec}  	
For a closed orientable $GRS^+$ $(M^{n}, g,f,\tau)$ the second variation of the $\nu$-entropy for any pertubation direction $h \in {C^\infty }({S^2}({T^*}M))$ is given by
  	\begin{eqnarray}
  		{\nu _g}^{''}(h)= \int_{M} { < N(h),h > dm}
  	\end{eqnarray}
  	here 
  	$$ N(h) = \frac{1}{2}{\Delta _f}h + Rm(h, - ) + div_f^\dag di{v_f}h+ \frac{1}{2}{\nabla ^2}{{\upsilon}_h} - Ric\frac{{\int_{M} { < Ric,h > dm} }}{{\int_{M}  {Rdm} }}$$
  	$$ =\frac{1}{2}{\Delta_{f,L}}h +\frac{1}{2\tau}h + div_f^\dag di{v_f}h +\frac{1}{2}{\nabla ^2}{{\upsilon}_h} - Ric\frac{{\int_{M} { < Ric,h > dm} }}{{\int_{M}  {Rdm} }}$$
  	in this experssion, ${\upsilon}_h$ is unique solution of the equation
  	\begin{eqnarray}
  		{\Delta _f}{\upsilon _h} + \frac{1}{{2\tau }}{\upsilon_h} = di{v_f}di{v_f}h  ,\int_{M}{\upsilon _h}dm=0
  	\end{eqnarray}
  	Here N is a  self-adjoint degenerate elliptic operator.
  \end{theorem}
\begin{proof}

	 see\cite{cao zhu} Theorem(1.1) and \cite{hail}.
	\end{proof}
 \subsection{Useful Formulas}
		\begin{theorem}\label{us}
		For a closed orientable $GRS^+$ $(M^{n}, g,f,\tau)$ we have
		\begin{eqnarray}
     	 Ric(\nabla f,-) = \displaystyle\frac{1}{2}dR & i.e \ \ div_{f}Ric=0\nonumber\\
		 g^{pq}\nabla_{p} R_{qjkl} =g^{pq} R_{pjkl}\nabla_{q} f &i.e\ \  div_{f}Rm=0\nonumber\\
		 \Delta_{f} Ric+2Rm(h,-) =\displaystyle\displaystyle\frac{1}{\tau}Ric& i.e \ \ \Delta_{f,L}Ric=0\nonumber\\
		 \Delta_{f}R  = \displaystyle\frac{1}{\tau}R-2|Ric|^2 & \nonumber\\ \Delta_{f}f  =  \displaystyle-\frac{1}{\tau}f+Const&\nonumber \\
		  \displaystyle\int_{M}{R dm}   =  \displaystyle 2\tau\int_{M}{|Ric|^2 dm} &\nonumber
		\end{eqnarray}
	\end{theorem}
     \begin{proof}
	see \cite{wyl} Lemma(2.1).
      	\end{proof}
       \begin{theorem}\label{spec}
      	For a closed orientable $GRS^+$ $(M^{n}, g,f,\tau)$, the first eigenvalue of the weighted Laplacian on functions $\Delta_{f}$ is stricly less than $-\frac{1}{2\tau}$.
      	$$(\lambda_{1} <-\frac{1}{2\tau})$$
    \end{theorem}
\begin{proof}
      	see \cite{cao zhu} page 9.
          \end{proof}
       From this theorem it follows that ${\upsilon}_h$ function in the second variation formula of the $\nu$-entropy exists and is unique.
  \section{Commutator Formulas}
  In this  section we will prove  commutator formulas between $div^\dag_{f},div_{f},\Delta_{f},\Delta_{f,L}$ on a closed orientable $GRS^+$. 
  Stability operator of the $\nu$-entropy has complicated formula, this operator is related to the weighted Lichnerowicz Laplacian.
  One of  the difficulties in computation of the second variation for a given pertubation direction, is computaion of an unknown function $\upsilon_{h}$ whose  Hessian  is in the stability operator. This function satisfies  the second order differential equation
  \begin{eqnarray}
  {\Delta _f}{\upsilon _h} + \frac{1}{{2\tau }}{\upsilon_h} = di{v_f}di{v_f}h \nonumber
  \end{eqnarray}
  note that from  Theorem (\ref{spec}) this function is unique,
  in some situation for example relation between linear stability and dynamical stability, from Ebin-Berger decomposition  Theorem  (see \cite{berger} Collorary(4.1) ),  tangent space to metric $g$,  i.e ${C^\infty }({S^2}({T^*}M))$  decomposites  to two orthogonal subspaces
  \begin{eqnarray}
  {C^\infty }({S^2}({T^*}M)) = Ker(di{v_f}) \oplus Im (div_f^\dag )\nonumber 
  \end{eqnarray}
 Then because $\nu$-entropy is invariant under the scaling and action of diffeomorphism group $Diff(M)$ on metric $g$ and since tangent vector to action of diffeomorphism group on metric is  ${L_{\# \omega }}g =-2div_f^\dag\omega , \omega \in \Omega^{1}(M)$, therefore the second variation of the $\nu$-entropy on $Im (div_f^\dag )$   is zero, hence in this situation we can assume that  $div_f h=0$ so $\upsilon_h=0$.
  In general  except for a few special cases  we have to  find unstability direction case by case and state by state, therefore we have to find $\upsilon_{h}$.
  Now because it is difficult to find $\upsilon_{h}$, we understand that  if we take $\upsilon_{h}=div_{f}div_{f}k$ for unknown tensor $k\in C^\infty ({S^2}({T^*}M))$, then we have $$\Delta_{f}div_{f}div_{f}k+\frac{1}{2\tau}div_{f}div_{f}k=div_{f}div_{f}h$$  now if we find commutator between $div_{f}div_{f}$ and $\Delta_{f}$, then we get a better understanding  of relation beatween $\upsilon_{h}$ and $h$. After this  we  find our commutator formulas and specially we find $div_fdiv_{f}\Delta_{f,L}k=\Delta_{f}div_{f}div_{f}k$ for $k \in C^\infty({S^2}({T^*}M))$ (indeed this is why that we found our commutator formulas ).  
  Now if $\upsilon_{h}=div_{f}div_{f}k$, then we have
     $$div_{f}div_{f}(\Delta_{f,L}k+\frac{1}{2\tau}k)=div_{f}div_{f}h$$
   Therefore if  $h=\Delta_{f,L}k+\frac{1}{2\tau}k$,  then  $\upsilon_{h}=div_{f}div_{f}k$. Now whenever we work with $k$ instead $h$ i.e given  $k \in C^\infty({S^2}({T^*}M))$,  then $h=\Delta_{f,L}k+\frac{1}{2\tau}k$
and $\upsilon_{h}$ is found. In the other hand we start with  $\upsilon_{h}$ as $\upsilon_{h}=div_{f}div_{f}k$, and then we find an $h\in  C^\infty({S^2}({T^*}M))$for that.
  After we have found our commutator formulas we found that Deruelle, Alix  \cite{alex} already obtained one of  our commutator formulas with only a time derivative difference(Theorem\ref{divlich}), indeed the order and method of our initial proof of  our results  is almost exactly the same as the method of Deruelle. Here we give another order and  proof for our commutator formulas. The point is that we have found them without knowing that Deruelle had already reached to these formulas.  
   \begin{theorem}\label{divda}
   	For a closed orientable $GRS^+$ $(M^{n}, g,f,\tau)$and any smooth function $a \in C^{\infty}(M)$ we have
 	\begin{eqnarray}
 		{\Delta _f}da = d{\Delta _f}a +\frac{1}{{2\tau }}da
 	\end{eqnarray}
  \end{theorem}
  \begin{proof}
  	
 	\begin{eqnarray}
 	{\Delta _f}da - d{\Delta _f}a &=& \Delta da - {\nabla _{\nabla f}}da - (d\Delta a - d < \nabla a,\nabla f > ) =\nonumber \\
 	&=& \Delta da - {\nabla ^2}a(\nabla f, - ) - d\Delta a + {\nabla ^2}a(\nabla f, - ) + {\nabla ^2}f(\nabla a, - )\nonumber \\
 	&=& \Delta da - d\Delta a + {\nabla ^2}f(\nabla a, - ) = (Ric + {\nabla ^2}f)(\nabla a, - ) =\nonumber\\
 	&=&\frac{1}{{2\tau }}da\nonumber 
 \end{eqnarray}
\end{proof}
   \begin{theorem}\label{divomega}
       	For a closed orientable $GRS^+$ $(M^{n}, g,f,\tau)$ and $\omega\in\Omega^{1}(M)$ we have
     	\begin{eqnarray}
		di{v_f}{\Delta _f}\omega =  {\Delta _f}di{v_f}\omega  + \frac{1}{{2\tau }}di{v_f}\omega \nonumber
	    \end{eqnarray}
        \end{theorem}
    \begin{proof}
    	
	    Take \ \ $a=di{v_f}{\Delta _f}\omega -  {\Delta   _f}di{v_f}\omega  - \frac{1}{{2\tau }}di{v_f}\omega $, \ Now we have  
	   \begin{eqnarray}
		\int_{M} a^2 dm&=&\nonumber\\
		&=&\int_{M} |di{v_f}{\Delta _f}\omega  - {\Delta _f}di{v_f}\omega  - \frac{1}{{2\tau }}di{v_f}\omega|^{2}dm\nonumber\\
		&=&\int_{M}< di{v_f}{\Delta _f}\omega  - {\Delta _f}di{v_f}\omega  - \frac{1}{{2\tau }}di{v_f}\omega, a>dm\nonumber\\
		&=&\int_{M}<\omega, - {\Delta _f}da + d{\Delta _f}a + \frac{1}{{2\tau }}da>dm \nonumber\\
		&=&0\nonumber
	       \end{eqnarray}	
	Now since $\int_{M} a^2  \ dm=0$ and $M$ is compact, therefore we conclude that $a=0$  
	i.e 
	$di{v_f}{\Delta _f}\omega =  {\Delta _f}di{v_f}\omega  + \frac{1}{{2\tau }}di{v_f}\omega$
	\end{proof}
     \begin{theorem}\label{lielap}
	For a closed orientable $GRS^+$$(M^{n}, g,f,\tau)$ and  $\omega\in\Omega^{1}(M)$ we have
	\begin{eqnarray}
		{\Delta _{f,L}}({\mathcal{L}_{\# \omega }}g)+ \frac{1}{2\tau }{\mathcal{L}_{\# \omega }}g  &=& {\mathcal{L}_{\# ({\Delta _f}\omega )}}g 
	  \end{eqnarray}
        \end{theorem}
    
  \begin{proof}
   	
   	\begin{eqnarray}
	{\Delta _f}({\mathcal{L}_{\# \omega }}g) & =& \Delta ({\mathcal{L}_{\# \omega }}g) - \mathop \nabla \nolimits_{\nabla f} {\mathcal{L}_{\# \omega }}g \nonumber
\end{eqnarray}
	Now for the first term of the right hand side we have
\begin{eqnarray}
			{\Delta }({\mathcal{L}_{\# \omega }}g)_{ij} & =& {g^{pq}} \big[{\nabla _p}{\nabla _q}({\nabla _i}{\omega _j} + {\nabla _j}{\omega _i})  \big] \nonumber\\
	& = & {g^{pq}} \big[ {\nabla _p}({\nabla _i}{\nabla _q}{\omega _j} + {R_{qijs}}{\omega ^s} + {\nabla _j}{\nabla _q}{\omega _i}+ {R_{qjis}}{\omega ^s}) \nonumber \\
		& = & {g^{pq} } \big[ ({{\nabla _p}{\nabla _i}{\nabla _q}{\omega _j} + {\nabla _p}{\nabla _j}{\nabla _q}{\omega _i}}) \nonumber \\ & + &  ({\nabla _p}{R_{qijs}} {\omega ^s} + {\nabla _p}{R_{qjis}} {\omega ^s} )+    ( {{R_{qijs}} + {R_{qjis}}} ){\nabla _p}{\omega ^s}  \big] \nonumber\\
	& = &{g^{pq}}\big[  ( {\nabla _i}{\nabla _p}{\nabla _q}{\omega _j} + {g^{\alpha \beta }}({R_{piq\alpha }}{\nabla _\beta }{\omega _j} + {R_{pij\alpha }}{\nabla _q}{\omega _\beta }))\nonumber \\ &+&  ({\nabla _j}{\nabla _p}{\nabla _q}{\omega _i}+ {g^{\alpha \beta }}({R_{pjq\alpha }}{\nabla _\beta }{\omega _i} + {R_{pji\alpha }}{\nabla _q}{\omega _\beta }) )\nonumber\\ &+&  ({\nabla _p}{R_{qijs}} {\omega ^s} + {\nabla _p}{R_{qjis}} {\omega ^s} )+ \left( {{R_{qijs}} + {R_{qjis}}} \right){\nabla _p}{\omega ^s}\big]  \nonumber\\
	& = &{g^{pq}}\big[ ({\nabla _i}{\nabla _p}{\nabla _q}{\omega _j} + {\nabla _j}{\nabla _p}{\nabla _q}{\omega _i}) + ({\nabla _p}{R_{qijs}} {\omega ^s} + {\nabla _p}{R_{qjis}} {\omega ^s} )\nonumber\\ &-&2 ({R_{qisj}} + {R_{qjsi}}){\nabla _p}{\omega ^s} \big]+ {g^{\alpha \beta }}({R_{i\alpha }} {\nabla _\beta }{\omega _j} + {R_{j\alpha }} {\nabla _\beta }{\omega _i})\nonumber \\
	&=& ({\nabla _i}\Delta {\omega _j} + {\nabla _j}\Delta {\omega _i})
	+g^{pq} ({\nabla _p}{R_{qijs}} {\omega ^s} + {\nabla _p}{R_{qjis}} {\omega ^s} )\nonumber\\&-&{2g^{pq}}({R_{qisj}} + {R_{qjsi}}){\nabla _p}{\omega ^s} + {g^{\alpha \beta }}({R_{i\alpha }} {\nabla _\beta }{\omega _j} + {R_{j\alpha }} {\nabla _\beta }{\omega _i}) \nonumber
\end{eqnarray}
And for the second term we have
\begin{eqnarray}
    ( \nabla_{\nabla f} {\mathcal{L}_{\# \omega }}g)_{ij}
	 & =&  {g^{pq}} \big[ {\nabla _p}({\nabla _i}{\omega _j} + {\nabla _j}{\omega _i}){\nabla _q}f \big] \nonumber\\
	&=&g^{pq}\big[({\nabla _i}{\nabla _p}{\omega _j} + {R_{pijs}}{\omega ^s} + {\nabla _j}{\nabla _p}{\omega _i} + {R_{pjis}}{\omega ^s}){\nabla _q}f \big] \nonumber \\
	& = & {g^{pq} } \big[ ( {R_{pijs}}+ {R_{pjis}}){{\nabla _q}f} {\omega ^s} +     ( {{\nabla _i}{\nabla _p}{\omega _j} + {\nabla _j}{\nabla _p}{\omega _i}} ){\nabla _q}f \big] \nonumber\\
	 &=&g^{pq} \big[ \left( {{R_{qijs}} + {R_{qjis}}} \right){\nabla _p}f{\omega ^s} + \left( {{\nabla _i}({\nabla _p}{\omega _j}{\nabla _q}f) - {\nabla _i}{\nabla _q}f{\nabla _p}{\omega _j}} \right)\nonumber\\ &+& \left( {{\nabla _j}({\nabla _p}{\omega _i}{\nabla _q}f) - {\nabla _j}{\nabla _q}f{\nabla _p}{\omega _i}} \right) \big]\nonumber\\
	&=&{g^{pq}}\big[ ({R_{qijs}} + {R_{qjis}}){\nabla _p}f{\omega ^s} + ({\nabla _i}({\nabla _p}{\omega _j}{\nabla _q}f) + {\nabla _j}({\nabla _p}{\omega _i}{\nabla _q}f))\nonumber\\ &-& ({\nabla _i}{\nabla _q}f{\nabla _p}{\omega _j} + {\nabla _j}{\nabla _q}f{\nabla _p}{\omega _i}) \big]\nonumber
\end{eqnarray}
Note that
$$ {\nabla _i}{\nabla _p}{\omega _j} {\nabla _q}f=
	\left( {{\nabla _i}({\nabla _p}{\omega _j}{\nabla _q}f) - {\nabla _i}{\nabla _p}f{\nabla _q}{\omega _j}} \right)$$
Now   if we   substitute these   two  calculated terms and considering

\begin{eqnarray}
 	{\Delta _f}({\mathcal{L}_{\# \omega }}g)_{ij}&=& ({\nabla _i}\Delta {\omega _j} + {\nabla _j}\Delta {\omega _i}) - ({\nabla _i}({\nabla _{\nabla f}}\omega ) + {\nabla _j}({\nabla _{\nabla f}}\omega ))\nonumber \\& +& \big[ ({g^{\alpha \beta }}({R_{i\alpha }} + {\nabla _i}{\nabla _\alpha }f){\nabla _\beta }{\omega _j} + {g^{\alpha \beta }}({R_{j\alpha }} + {\nabla _j}{\nabla _\alpha }f){\nabla _\beta }{\omega _i})\big]\nonumber\\& -& 2{g^{pq}}({R_{qisj}} + {R_{siqj}}){\nabla _p}{\omega ^s} +g^{pq}(\nabla_{p}R_{qijs}-R_{pijs}\nabla_{q}f)\nabla_{p}\omega^{s} \nonumber\\
 	&+&g^{pq}(\nabla_{p}R_{qjis}-R_{pjis}\nabla_{q}f)\nabla_{p}\omega^{s}\nonumber
 \end{eqnarray}
Now according to Theorem(\ref{us}) we have $div_{f}Rm=0$ therefore 
\begin{eqnarray}
	{\Delta _f}({\mathcal{L}_{\# \omega }}g)_{ij}&=&({\nabla _i}\Delta {\omega _j} + {\nabla _j}\Delta {\omega _i}) - ({\nabla _i}({\nabla _{\nabla f}}\omega )_{j} + {\nabla _j}({\nabla _{\nabla f}}\omega )_{i})\nonumber \\& +& \big[ ({g^{\alpha \beta }}({R_{i\alpha }} + {\nabla _i}{\nabla _\alpha }f){\nabla _\beta }{\omega _j} + {g^{\alpha \beta }}({R_{j\alpha }} + {\nabla _j}{\nabla _\alpha }f){\nabla _\beta }{\omega _i})\big]\nonumber\\& -& 2{g^{pq}}({R_{qisj}} + {R_{siqj}}){\nabla _p}{\omega ^s}\nonumber
\end{eqnarray}
On the other hand we have
\begin{eqnarray}
		2{g^{pq}}({R_{qisj}} + {R_{siqj}}){\nabla _p}{\omega ^s}&=&2{g^{pq}}{g^{rs}}(R_{qisj}+R_{siqj})\nabla_{p}\omega_{r}\nonumber\\
		&=&2{g^{pq}}{g^{rs}}R_{qisj}(\nabla_{p}\omega_{r}+\nabla_{r}\omega_{p})\nonumber\\
		&=&2Rm({\mathcal{L}_{\# \omega }}g,-)_{ij}\nonumber
\end{eqnarray}
and finally
\begin{eqnarray}
	{\Delta _f}({\mathcal{L}_{\# \omega }}g)&=&{\mathcal{L}_{\# \Delta \omega }}g -{\mathcal{L}_{\# {\nabla _{\nabla f}}\omega }}g + \frac{1}{{2\tau }}{\mathcal{L}_{\# \omega }}g - 2Rm({\mathcal{L}_{\# \omega }}g, - ) \nonumber\\
	&=& {\mathcal{L}_{\# {\Delta _f}\omega }}g + \frac{1}{{2\tau }}{\mathcal{L}_{\# \omega }}g - 2Rm({\mathcal{L}_{\# \omega }}g, - ) \nonumber
\end{eqnarray}
Therefore we conclude that

\begin{eqnarray}
\big({\Delta _f}({\mathcal{L}_{\# \omega }}g) + 2Rm({\mathcal{L}_{\# \omega }}g, - ) - \frac{1}{\tau }{\mathcal{L}_{\# \omega }}g \big)+ \frac{1}{2\tau }{\mathcal{L}_{\# \omega }}g  &=& {\mathcal{L}_{\# {\Delta _f}\omega }}g \nonumber
\end{eqnarray}
 And finally it follows that
$${\Delta _{f,L}}({\mathcal{L}_{\# \omega }}g)+ \frac{1}{2\tau }{\mathcal{L}_{\# \omega }}g  = {\mathcal{L}_{\# ({\Delta _f}\omega )}}g $$
\end{proof}

\begin{corollary}\label{divshoplich}
	For a closed orientable $GRS^+$ and  $\omega\in\Omega^{1}(M)$we have
	\begin{eqnarray}
		div^\dag_{f}\Delta_{f}\omega=\Delta_{f,L}div^\dag_{f}\omega+\frac{1}{2\tau}div^\dag_{f}\omega
	\end{eqnarray}
\end{corollary}

\begin{theorem}\label{divlich}
	For a closed orientable $GRS^+$ $(M^{n}, g,f,\tau)$ and  $h \in {C^\infty }({S^2}({T^*}M))$  we have
	\begin{eqnarray}
	\Delta_{f}div_{f}h=div_{f}\Delta_{f,L}h+\frac{1}{2\tau}div_{f}h
	\end{eqnarray}
\end{theorem}

\begin{proof}
Take $\omega=\Delta_{f}div_{f}h-div_{f}\Delta_{f,L}h-\frac{1}{2\tau}div_{f}h$,  now we have 

	\begin{eqnarray}
	 \ \ \int_{M} |\omega|^2 \ dm&=&\int_{M} <\Delta_{f}div_{f}h-div_{f}\Delta_{f,L}h-\frac{1}{2\tau}div_{f}h
	,\omega>dm\nonumber\\
	&=&\int_{M}<h,div_{f}^\dag\Delta_{f}\omega-\Delta_{f,L}div_{f}^\dag \omega-\frac{1}{2\tau}div_{f}^\dag\omega> dm\nonumber\\
	&=&0\nonumber
\end{eqnarray}
Now $\displaystyle\int_{M}|\omega|^2 \ dm=0$,  therefore  because $M$ is compact it follows that  $\omega=0$.

\end{proof}

  \begin{theorem}\label{divdivlich}
    For a closed orientable $GRS^+$ $(M^{n}, g,f,\tau)$ and  $h \in {C^\infty }({S^2}({T^*}M))$ we have  	
	\begin{eqnarray}
		\Delta_{f}div_{f}div_{f}h=div_{f}div_{f}\Delta_{f,L}h
	\end{eqnarray}
	\end{theorem}

\begin{proof}

  By Theorem(\ref{divlich}) it follows that

 \begin{eqnarray}
	\Delta_{f}div_{f}h&=&div_{f}\Delta_{f,L}h+\frac{1}{2\tau}div_{f}h\nonumber\\
	&\ \ & and  \nonumber\\
	div_{f}\Delta_{f}div_{f}h&=&div_{f}div_{f}\Delta_{f,L}h+\frac{1}{2\tau}div_{f}div_{f}h\nonumber
  \end{eqnarray}
  On the other hand because $div_{f}h$ is a differential form therefore according to Theorem(\ref{divomega}) if we take $\omega=div_{f}h$ we have
  \begin{eqnarray}
  		div_{f}\Delta_{f}div_{f}h&=\Delta_{f}div_{f}div_{f}h+\frac{1}{2\tau}div_{f}div_{f}h\nonumber
  \end{eqnarray} 
 and finally $\Delta_{f}div_{f}div_{f}h=div_{f}div_{f}\Delta_{f,L}h$
 \end{proof}

 \newtheorem{shopldivlich}{Corollary}[section]
 \begin{theorem}\label{shopdivlich}
	For a closed orientable $GRS^+$ $(M^{n}, g,f,\tau)$ and  $h \in {C^\infty }({S^2}({T^*}M))$
	\begin{eqnarray}
	div^\dag_{f}div_{f}\Delta_{f,L}h=	\Delta_{f,L}div^\dag_{f}div_{f}
	\end{eqnarray}
  \end{theorem}
  \begin{proof}
  	
  	From Theorem(\ref{divlich}) we have
  \begin{eqnarray}
	div^\dag_{f}div_{f}\Delta_{f,L}h&=&div^\dag_{f}(\Delta_{f}div_{f}h-\frac{1}{2\tau}div_{f}h)\nonumber\\
	&=&div^\dag_{f}\Delta_{f}div_{f}h-\frac{1}{2\tau}	div^\dag_{f}div_{f}h\nonumber
\end{eqnarray}
		Now  from   Corollary(\ref{divshoplich})  we   have 
\begin{eqnarray}
	div^\dag_{f}div_{f}\Delta_{f,L}h&=&\Delta_{f,L}	div^\dag_{f}div_{f}h+\frac{1}{2\tau}	div^\dag_{f}div_{f}h-\frac{1}{2\tau}	div^\dag_{f}div_{f}h\nonumber\\
&=&\Delta_{f,L}	div^\dag_{f}div_{f}h\nonumber
\end{eqnarray}
\end{proof}

\section{Kernel of Stability Operator}
 We know that $\nu$-entropy is invariant with respect to action of  diffeomorfism group Diff(M), now  because tangent vector to action of diffeomorphism group on metrics is  ${L_{\# \omega }}g =-2div_f^\dag\omega , \omega \in \omega(M)$, therefore the second variation of the $\nu$-entropy on $Im (div_f^\dag )$   is zero $$<N(div^\dag_{f}\omega),div^\dag_{f}\omega>dm= \displaystyle\int_{M}<N({\mathcal{L}_{X }}g),{\mathcal{L}_{X }}g>  dm =0$$  But is this true that for $h\in Im(div^\dag_{f})$,  $N(h)=0$? 
Cao and He in \cite{sym} proved that for trivial $GRS^+$ (Einstein metric with positive scalar curvature) $Im(div^\dag_{f})\subset Ker N$.  In this section we extend this result  to nontrivial
$GRS^+$.

For computation of the stability operator on  $Im(div^\dag_{f})$ we need to compuate $\upsilon_{h},div_{f}div^\dag_{f}$ on  $Im(div^\dag_{f})$.
For this purpose  we need to prove some theorems.

     \begin{theorem}\label{divlie}
  	For a closed orientable $GRS^+$ $(M^{n}, g,f,\tau)$ we have 
  	 \begin{eqnarray}
  			di{v_f}({\mathcal{L}_{\# \omega }}g)={\Delta _f}\omega  + d(di{v_f}\omega ) + \frac{1}{{2\tau }}{\omega}
  	 \end{eqnarray}
       \end{theorem}
   
\begin{proof}

	Because $	di{v_f}({\mathcal{L}_{\# \omega }}g)_{i} = div({\mathcal{L}_{\# \omega }}g) - {\mathcal{L}_{\# \omega }}g(\nabla f, - ) $ therefore
  	\begin{eqnarray}
  		di{v_f}({\mathcal{L}_{\# \omega }}g)_{i} &=& {g^{pq}}[{\nabla _p}({\nabla _q}{\omega _i} + {\nabla _i}{\omega _q}) - ({\nabla _p}{\omega _i} + {\nabla _i}{\omega _p}){\nabla _q}f]\nonumber \\
  		&=& {g^{pq}}[{\nabla _p}{\nabla _q}{\omega _i} + {\nabla _p}{\nabla _i}{\omega _q}  - {\nabla _p}{\omega _i}{\nabla _q}f - {\nabla _i}{\omega _p}{\nabla _q}f] \nonumber 
  	\end{eqnarray}
  we have $${\nabla _p}{\nabla _i}{\omega _q} = {\nabla _i}{\nabla _p}{\omega _q} + {R_{piqs}}{\omega ^s}$$ Therefore
 
  		$$di{v_f}({L_{\# \omega }}g)_{i}= {g^{pq}}[{\nabla _p}{\nabla _q}{\omega _i} + {\nabla _i}{\nabla _p}{\omega _q} + {R_{piqs}}{\omega ^s} - {\nabla _p}{\omega _i}{\nabla _q}f - {\nabla _i}{\omega _p}{\nabla _q}f] $$
 Now since  			  		 $$ {\nabla _i}{\omega _p}{\nabla _q}f= {\nabla _i}({\nabla _q}f{\omega _p}) - {\nabla _i}{\nabla_q}f{\omega _p} $$ therefore
  		 	\begin{eqnarray}
  			di{v_f}({L_{\# \omega }}g)_{i} &=& {g^{pq}}[{\nabla _p}{\nabla _q}{\omega _i} + {\nabla _i}{\nabla _p}{\omega _q} + {R_{is}}{\omega ^s}- {\nabla _p}{\omega _i}{\nabla _q}f - {\nabla _i}({\omega _p}{\nabla _q}f)\nonumber\\ &+&    {\nabla _i}{\nabla_q}f{\omega _p}]\nonumber \\
  		&=& {g^{pq}}[({\nabla _p}{\nabla _q}{\omega _i} - {\nabla _p}{\omega _i}{\nabla _q}f) + {\nabla _i}({\nabla _p}{\omega _q} - {\omega _p}{\nabla _q}f)+(R_{pi}+{\nabla_{p}\nabla_{i}}f)\omega_q]\nonumber  
  	  	\end{eqnarray}
    And finally we conclude that
  $$di{v_f}({L_{\# \omega }}g)= {\Delta _f}\omega  + d(di{v_f}\omega ) + \frac{1}{{2\tau }}{\omega}$$
  
  \end{proof}

 \begin{corollary}\label{divdivshop}
 		For a closed orientable $GRS^+$ and $\omega\in\Omega^{1}(M)$ we have 
 	\begin{eqnarray}
 		di{v_f}div_f^\dag \omega  =  - \frac{1}{2}({\Delta _f}\omega  + d(di{v_f}\omega ) + \frac{1}{{2\tau }}\omega )
 		\end{eqnarray}
  \end{corollary}

\begin{lemma}\label{shopdivshop}
		For a closed orientable $GRS^+$ $(M^{n}, g,f,\tau)$ and $\omega\in\Omega^{1}(M)$ we have 
	
	\begin{eqnarray}
		div_f^\dag di{v_f}(div_f^\dag \omega )& =&\frac{{ - 1}}{2}{\Delta _{f,L}}div_f^\dag \omega  - \frac{{1}}{{2\tau }}div_f^\dag \omega  + \frac{1}{2}{\nabla ^2}di{v_f}\omega \
	\end{eqnarray}
\end{lemma}
\begin{proof}
	From Corollary(\ref{divdivshop}) we have
\begin{eqnarray}
	div_f^\dag di{v_f}(div_f^\dag \omega )& =&  - \frac{1}{2}div_f^\dag ({\Delta _f}\omega  + d(di{v_f}\omega ) + \frac{1}{{2\tau }}\omega )\nonumber\\ &=& \frac{{ - 1}}{2}di{v^\dag }{\Delta _f}\omega  + \frac{1}{2}{\nabla ^2}di{v_f}\omega  - \frac{1}{{4\tau }}div_f^\dag \omega  \nonumber 
\end{eqnarray}
 From  Corollary(\ref{divshoplich})  we  have 
 	$$ div^\dag_{f}\Delta_{f}\omega=\Delta_{f,L}div^\dag_{f}\omega+\frac{1}{2\tau}div^\dag_{f}\omega $$
 	 Therefore
 	 \begin{eqnarray}
 	 			div_f^\dag di{v_f}(div_f^\dag \omega )&=& \frac{{ - 1}}{2}{\Delta _{f,L}}div_f^\dag \omega  - \frac{1}{{4\tau }}div_f^\dag \omega  + \frac{1}{2}{\nabla ^2}di{v_f}\omega  - \frac{1}{{4\tau }}div_f^\dag \omega  \nonumber\\
	&= &\frac{{ - 1}}{2}{\Delta _{f,L}}div_f^\dag \omega  - \frac{1}{{2\tau }}div_f^\dag \omega  + \frac{1}{2}{\nabla ^2}di{v_f}\omega \nonumber
\end{eqnarray}
\end{proof}

\begin{lemma}\label{divdivdivshop}
	For a closed orientable $GRS^+$ $(M^{n}, g,f,\tau)$ and $\omega\in\Omega^{1}(M)$ we have
	\begin{eqnarray}
		di{v_f}di{v_f}(div_f^\dag \omega )= -( {\Delta _f}di{v_f}\omega+ \frac{1}{{2\tau }}di{v_f}\omega ) 
	\end{eqnarray}
\end{lemma}
\begin{proof}
 From  Corollary (\ref{divdivshop})   we  have
 $$	di{v_f}di{v_f}(div_f^\dag \omega ) = - \frac{1}{2}di{v_{f}}({\Delta _f}\omega  + d(di{v_f}\omega ) + \frac{1}{{2\tau }}\omega )$$

	 Now  according to   Theorem(\ref{divomega})  we  have
	 \begin{eqnarray}
di{v_f}di{v_f}(div_f^\dag \omega ) &=&  -\frac{1}{2}({\Delta _f}di{v_f}\omega  + \frac{1}{{2\tau }}di{v_f}\omega  + {\Delta _f}di{v_f}\omega  + \frac{1}{{2\tau }}di{v_f}\omega )\nonumber\\
	&=& -( {\Delta _f}di{v_f}\omega+ \frac{1}{{2\tau }}di{v_f}\omega )\nonumber  
\end{eqnarray}
\end{proof}

\begin{theorem}\label{kernel}
		For a closed orientable $GRS^+$ $(M^{n}, g,f,\tau)$ and $\omega\in\Omega^{1}(M)$ we have 
		\begin{eqnarray}
			N(div^\dag_{f}\omega)=0
		\end{eqnarray}
	\end{theorem}
\begin{proof}	

Stability operator of the $\nu$-entropy has four non trivial  terms $\Delta_{f,L},div^\dag_{f}div_{f},\nabla^{2}\upsilon_{h}$ and finally coefficient of Ricci tensor, at  first we calculate all of these terms   on $Im( div^\dag_{f})$ seperatively, for this we need  to calculate
$$\Delta_{f,L}div_f^\dag \omega,	div_f^\dag di{v_f}(div_f^\dag \omega ),di{v_f}di{v_f}(div_f^\dag \omega )$$
		 For computation of $\upsilon_{h}$, from  Lemma(\ref{divdivdivshop}) it follows that $$di{v_f}di{v_f}(div_f^\dag \omega )= -( {\Delta _f}di{v_f}\omega+ \frac{1}{{2\tau }}di{v_f}\omega )$$  therefore according to uniquness of $\upsilon_h$, we conclude that for $h=div^\dag_{f}{\omega}$,  $\upsilon_{h}=-div_{f}\omega$, therefore we have \begin{eqnarray}
	\nabla^{2}{\upsilon_{h}}=-\nabla^{2}{div_{f}\omega}\nonumber
\end{eqnarray} 
For $div^\dag_{f}div_{f}h$ from Lemma(\ref{shopdivshop})  it follows that
$$div_f^\dag di{v_f}(div_f^\dag \omega )=\frac{{ - 1}}{2}{\Delta _{f,L}}div_f^\dag \omega  - \frac{{1}}{{2\tau }}div_f^\dag \omega  + \frac{1}{2}{\nabla ^2}di{v_f}\omega $$
 Furthermore we have
 \begin{eqnarray}
 	\displaystyle\int_{M}<Ric,div^\dag_{f}\omega>dm&=&\displaystyle \int_{M}<div_{f}(Ric),\omega>dm\nonumber\\
 	&=&0\nonumber
 \end{eqnarray}
Now if we  substitute these terms in the experssion of the stability operator,  then we have
\begin{eqnarray}
	N(h) &=&\frac{1}{2}{\Delta_{f,L}}div^\dag_{f}\omega +\frac{1}{2\tau}div^\dag_{f}\omega + div_f^\dag di{v_f}div^\dag_{f}\omega +\frac{1}{2}{\nabla ^2}{{\upsilon}_h} - Ric\frac{{\int_{M} { < Ric,div^\dag_{f}\omega > dm} }}{{\int_{M}  {Rdm} }}\nonumber\\
	&=&\frac{1}{2}{\Delta_{f,L}}div^\dag_{f}\omega +\frac{1}{2\tau}div^\dag_{f}\omega + div_f^\dag di{v_f}div^\dag_{f}\omega +\frac{1}{2}{\nabla ^2}{{\upsilon}_h} \nonumber\\
	&=&\frac{1}{2}{\Delta_{f,L}}div^\dag_{f}\omega +\frac{1}{2\tau}div^\dag_{f}\omega+\frac{{ - 1}}{2}{\Delta _{f,L}}div_f^\dag \omega  - \frac{{1}}{{2\tau }}div_f^\dag \omega  + \frac{1}{2}{\nabla^2}di{v_f}\omega-\frac{1}{2}\nabla^{2}{div_{f}\omega}\nonumber\\
	&=&0\nonumber
  \end{eqnarray}
\end{proof}

\section{Main Results}

Cao and Zhu in \cite{cao zhu} using  similarity between  shrinking soliton of mean curvature flow and $GRS^+$ of Ricci flow and work of  Colding and Minicozzi in linear stability of  shrinking soliton of mean curvature flow,  have shown
that  for linear stability of $GRS^+$ metrics it is neccesary that the first eigenvalue of the weighted  Lichnerowicz Laplace operator (restricted to transversal tensor i.e $div_{f}h=0$) is zero with  multiplicity one  and with Ric being an eigentensor  $Ric$. In this section we will extend  their result.  First  we eliminate condition $(h\in Ker(div_{f}))$ and we replace zero bound with a weaker bound. Secondly we find the sufficient condition for linear stability of a $GRS^+$. 
Finally we find some relations between eigentensors of the stability operator $N$ and the weighted Lichnerowics Laplacian $\Delta_{f,L}$.

\newtheorem{nec}{Theorem}[section]
\begin{nec}\label{nec}
	  A necessary condition for linear stability	 of closed oraiantable $GRS^+$ $(M^{n}, g,f,\tau)$ is that the first eigenvalue of  the weighted Lichnerowicz Laplacian $\Delta_{f,L}$ (except zero with one multiplicity and Ricci tensor  $Ric$ as eigentensor) is not greather than $-\frac{1}{2\tau}$.
	\end{nec}
	  
	  \begin{proof}  
	  Suppose that  $\Delta_{f,L}h=\lambda h$ for $h\in  C^\infty({S^2}({T^*}M)),\lambda\in \mathbb{R}$ and $\lambda> -\frac{1}{2\tau}$.\\
	  First suppose that $\lambda \not=0$, 
	   we compute the second variation of the $\nu$-entropy in the direction $h$, for this purpose we should  compute $N(h)$.  Similar to privious section we  have to compute four non trivial  terms $\Delta_{f,L},div^\dag_{f}div_{f},\nabla^{2}\upsilon_{h}$ and finally coefficient of Ricci tensor,	  
	   first we computate $\upsilon_{h}$. Because $\Delta_{f,L}h=\lambda h$, therefore by Theorem(\ref{divdivlich}) we have
	  $$	\Delta_{f}div_{f}div_{f}h=div_{f}div_{f}\Delta_{f,L}h=\lambda div_{f}div_{f}h$$
	  And therefore $$	\Delta_{f}div_{f}div_{f}h+\frac{1}{2\tau}div_{f}div_{f}h=(\lambda+\frac{1}{2\tau})div_{f}div_{f}h$$  so  $\lambda+\frac{1}{2\tau}$ is an eigenvalue of $\Delta_{f}$ and $div_{f}div_{f}h$
       is an eigenfunction of $\Delta_{f}$, but according to espectral estimate of Theorem (\ref{spec})it follows that $\lambda< -\frac{1}{2\tau}$, therefore  $div_{f}div_{f}h=0$ and according to  uniquness of $\upsilon_{h}$, $\upsilon_{h}=0$.
           On the other hand we assume that $\lambda\not=0$ therefore
  \begin{eqnarray}
  	\displaystyle{\int_{M}}<Ric,h>dm&=&\int_{M}<Ric,{\frac{\lambda}{\lambda}}  \  h>dm\nonumber\\
  	&=&\int_{M}<Ric,{\frac{1}{\lambda}}  \Delta_{f,L} h>dm\nonumber\\
    	&=&\int_{M}<\Delta_{f,L}Ric,{\frac{1}{\lambda}}h>dm\nonumber\\
    	&=&0\nonumber
  \end{eqnarray} 
From Theorem (\ref{us}) we have $\Delta_{f,L}Ric=0$, therefore the coefficient of  Ricci tensor in the experssion of the stability operator is zero.
$$	\displaystyle{\int_{M}}<Ric,h>dm=0$$	  
Now we have
\begin{eqnarray}
	{\nu _g}^{''}(h)= \displaystyle\int_{M} { < N(h),h > dm}\nonumber
\end{eqnarray}
here 
$$N(h) =\frac{1}{2}{\Delta_{f,L}}h +\frac{1}{2\tau}h + div_f^\dag di{v_f}h +\frac{1}{2}{\nabla ^2}{{\upsilon}_h} - Ric\frac{{\int_{M} { < Ric,h > dm} }}{{\int_{M}  {Rdm} }}$$
In this expersion the coefficient of Ricci is zero thus
$$N(h) =(\frac{\lambda}{2} +\frac{1}{2\tau}) h+ div_f^\dag di{v_f}h +\frac{1}{2}{\nabla ^2}{{\upsilon}_h} $$
Therefore
\begin{eqnarray}
	{\nu _g}^{''}(h)&=& \displaystyle\int_{M} { < N(h),h > dm}\nonumber\\
	&=&\int_{M}<(\frac{\lambda}{2} +\frac{1}{2\tau})h + div_f^\dag di{v_f}h ,h>dm\nonumber\\
	&=&\int_{M}(\frac{\lambda}{2} +\frac{1}{2\tau})|h|^{2}dm
	+\int_{M} <div_f^\dag di{v_f}h ,h>dm\nonumber\\
	&=&\int_{M}(\frac{\lambda}{2} +\frac{1}{2\tau})|h|^{2}dm+\int_{M}|div_{f}h|^2dm\nonumber\\
	&\geqslant&\frac{1}{4\tau}\int_{M}|h|^{2}dm>0\nonumber
\end{eqnarray}
Therefore $h$ is an unstability direction and $GRS^+$ is linear unstable.

Now suppose that $\Delta_{f,L}h=0, \  \ h\not=\mu Ric , \mu\in \mathbb{R},\mu\not=0$, by Theorem (\ref{divdivlich}) it follows that 
$$\Delta_{f}div_{f}div_{f}h=div_{f}div_{f}\Delta_{f,L}h=0$$ from  the maximum principle it follows that $div_{f}div_{f}h=0$ and according to uniquness of $\upsilon_{h}$, we have $\upsilon_{h}=0$ 
Therefore we conclude that
\begin{eqnarray}
	{\nu _g}^{''}(h)&=& \displaystyle\int_{M} { < N(h),h > dm}\nonumber\\
	&=&\int_{M}< \frac{1}{2\tau}h + div_f^\dag di{v_f}h - Ric\frac{{\int_{M} { < Ric,h > dm} }}{{\int_{M}  {Rdm} }} ,h>dm\nonumber\\
	&=&\int_{M} \frac{1}{2\tau}|h|^{2}dm + | di{v_f}h|^{2}dm -\frac{({\int_{M} { < Ric,h > dm} })^2}{{\int_{M}  {Rdm} }} \nonumber\\
	&\geqslant&\int_{M} \frac{1}{2\tau}|h|^{2}dm  -\frac{({\int_{M} { < Ric,h > dm} })^2}{{\int_{M}  {Rdm} }} \ . \nonumber
\end{eqnarray}

From Theorem (\ref{us}) we know that $\displaystyle \int_{M} R dm=\displaystyle\int_{M}2\tau|Ric|^{2}  dm$, therefore 
\begin{eqnarray}
	{\nu _g}^{''}(h)&\geqslant&\int_{M} \frac{1}{2\tau}|h|^{2}dm  -\frac{({\int_{M} { < Ric,h > dm} })^2}{{2\tau\int_{M}  {|Ric|^{2}dm} }} \nonumber\\
	&\geqslant&\Big(\int_{M} \frac{1}{2\tau}|h|^{2}dm\int_{M} 2\tau  {|Ric|^{2}dm}  -({\int_{M} { < Ric,h > dm} })^2\Big)\frac{1}{{{2\tau\int_{M}  {|Ric|^{2}dm} }}} \nonumber\\
	&\geqslant&\Big(\int_{M} |h|^{2}dm\int_{M}  {|Ric|^{2}dm}  -({\int_{M} { < Ric,h > dm} })^2\Big)\frac{1}{{{2\tau\int_{M}  {|Ric|^{2}dm} }}}  \ . \nonumber
\end{eqnarray}
But according to Cauchy–Schwarz inequality we know that  last term is positive and is zero if and only if $h=\mu Ric , \mu\in \mathbb{R},\mu\not=0$ but we assume that  $h\not=\mu Ric , \mu\in \mathbb{R},\mu\not=0$ therefore $\nu^{''}_{g}(h)>0$ and soliton is unstable

\end{proof}

\begin{theorem}\label{suf}
	Suppose that for a closed oraiantable $GRS^+$ $(M^{n}, g,f,\tau)$ the first eigenvalue of the weighted Lichnerowicz Laplacian $\Delta_{f,L}$ is not greather than $-\frac{1}{\tau}$ (except zero with one multiplicity and Ricci tensor  $Ric$ as eigentensor)  , then soliton is linear stable
\end{theorem}
\begin{proof}
	The main idea of the proof is that for arbitrary  $h\in  C^\infty({S^2}({T^*}M))$
	 for simplification  of computation, we write $h$ in a common orthonormal eigenbasis of the two operators $\Delta_{f,L}$ and $\Delta_{f,L}+div^\dag_{f}div_{f}$ (we prove that this eigenbasis exists), then  we prove that the second variation of the $\nu$-entropy in the direction $h$ is nonpositive.	
	  
	 First we prove that   two operators $\Delta_{f,L}$ and $\Delta_{f,L}+div^\dag_{f}div_{f}$ have common orthogonal eigenbasis, for this purpose we  prove that  these two operators are diagonalizable.\newline
	  These two operators are strongly elliptic. Indeed 
	 suppose that we have a differential operator $L:C^{\infty}(E)\rightarrow C^{\infty}(F)$ between two vector bundles
	 $(E,F)$,  and  suppose that  the principal symbol of $L$ in the direction $\omega\in \Omega^{1}(M)$ is denoted by $\sigma_{\omega}(D)$( i.e we have $\sigma_{\omega}:E_{p}\rightarrow F_{p}$).
	 Now trivially the operator $\Delta_{f,L}$  which is a twisted Lichnerowicz  Laplacian, is strongly elliptic with principal symbol
	 $$\sigma_{\omega}(\Delta_{f,L})h={|\omega|^2}h,$$for principal symbol of   $\Delta_{f,L}+div^\dag_{f}div_{f}$, a simple computation shows that 
	 $$\sigma_{\omega}(\Delta_{f,L}+div^\dag_{f}div_{f})h={|\omega|^2}h+h(\sharp\omega,-)\otimes\omega+\omega\otimes h(\sharp\omega,-) $$
	 such that $h \in C^{\infty}(S^2(T^*(M)))$ and $\omega\in\Omega^{1}(M)$. Now we have
 	
 	  \begin{eqnarray}
 	 &&<\sigma_{\omega}(\Delta_{f,L}+div^\dag_{f}div_{f})h,h>\nonumber\\
 	 &=&<{|\omega|^2}h+h(\sharp\omega,-)\otimes\omega+\omega\otimes h(\sharp\omega,-),h>\nonumber\\
 	 &=&|\omega|^{2}|h|^{2}+2|h(\sharp\omega,-)|^{2}\nonumber\\
 	 &\geqslant&|\omega|^{2}|h|^{2}\nonumber\\
 	 &\geqslant&|\omega|^{2}|h|^{2}\nonumber
 	 \end{eqnarray}
	Therefore $(\Delta_{f,L}+div^\dag_{f}div_{f})$ is strongly elliptic.

Now  $\Delta_{f,L}$ and $\Delta_{f,L}+div^\dag_{f}div_{f}$ are self-adjoint strongly  elliptic operators and by compactness of
$M$(see \cite{kroncke}page 15 ) and according to spectral theory have a discrete set of eigenvalues
$\lambda_{1}>\lambda_{2}>\lambda_{3}\dots $and $\lambda_{n}\rightarrow-\infty$  and any
eigenvalue has finite multiplicity and eigentensors of different eigenvalues are orthogonal. These two operators extend to two continuous linear maps on completion of  $C^{\infty}(S^2(T^*(M)))$ and from elliptic regularity all  eigentensors of these operators are smooth, but by Theorem(\ref{shopdivlich}) these two operators commute,  therefore should have a common basis of eigentensors(from functional analysis we know that any two diagonalizable commutative continuous maps diagonalizable simultaneously), with Gram–Schmidt process we can construct an orthonormal  common basis of eigentensors of these operators.\\
Suppose that an arbitrary tensor $h\in C^{\infty}(S^2(T^*(M)))$ may be written, in an unique way in this basis, as  $$h=\displaystyle\Sigma^{\infty}_{i=1}\lambda_{i}h_{i}$$
Now we compute the second variation of the  $\nu$-entropy in the direction $h$.
First suppose that for an  eigentensor  $h_{i}\in C^{\infty}(S^2(T^*(M)))$ of $\Delta_{f,L}$ and $\Delta_{f,L}+div^\dag_{f}div_{f}$ we have 
\begin{eqnarray}
	\Delta_{f,L}h_{i}&=&\lambda_{i}h_{i}\nonumber\\
	\Delta_{f,L}h_{i}+div^\dag_{f}div_{f}h_{i}&=&\mu_{i}h_{i}\nonumber
\end{eqnarray}
Therefore $div^\dag_{f}div_{f}=(\lambda_{i}-\mu_{i})h_{i}$. Now If $\lambda_{i}\not=\mu_{i}$,  then  $$h_{i}=(\lambda_{i}-\mu_{i})^{-1}div^\dag_{f}div_{f}h_{i}$$  
Therefore $h_{i}\in Im (div^\dag_{f})$, and by Theorem(\ref{kernel}). $$N(h_{i})=0$$

Now suppose that $\lambda_{i}=\mu_{i}$, hence $div^\dag_{f}div_{f}h_{i}=(\lambda_{i}-\mu_{i})h_{i}=0$, therefore we conclude that \begin{eqnarray}
	\int_{M}|div_{f}{h_{i}}|^{2}dm&=&\int_{M}<div^\dag_{f}div_{f}h_{i},h_{i}>dm\nonumber\\&=&0\nonumber
\end{eqnarray}
Therefore it follows that $div_{f}{h_{i}}=0$ and hence $\upsilon_{h_{i}}=0$ i.e $$div^\dag_{f}div_{f}h_{i}+\frac{1}{2}\nabla^{2}\upsilon_{h_{i}}=0$$
For this reason that we assume that $\lambda\leqslant-\frac{1}{\tau}$ and $Ric$ is only a generator of eigenspace of the zero eigenvalue of $\Delta_{f,L}$, therefore if $(h_{i}\not=\lambda Ric, \lambda\in\mathbb{R},\lambda\not=0)$  we conclude that 
\begin{eqnarray}
	\displaystyle{\int_{M}}<Ric,h_{i}>dm&=&\int_{M}<Ric,{\frac{\lambda_{i}}{\lambda_{i}}}  \  h_{i}>dm\nonumber\\
	&=&\int_{M}<Ric,{\frac{1}{\lambda_{i}}}  \Delta_{f,L} h_{i}>dm\nonumber\\
	&=&\int_{M}<\Delta_{f,L}Ric,{\frac{1}{\lambda_{i}}}h_{i}>dm\nonumber\\
	&=&0\nonumber
\end{eqnarray} 
Now for the stability operator of $h_{i}$ we have
\begin{eqnarray}
	N(h_{i})&=&(\frac{\lambda_{i}}{2} +\frac{1}{2\tau})h_{i} \nonumber
\end{eqnarray}
then  we can write $h\in C^{\infty}(S^2(T^*(M)))$  in a new way in this basis, as 
$$h=\displaystyle c_{0}Ric+\Sigma^{\infty}_{\lambda_{i}=\mu_{i}}c_{i}h_{i}+\Sigma^{\infty}_{\lambda_{j}\not=\mu_{j}}c_{j}h_{j}$$

Therefore it follows that
\begin{eqnarray}
	N(h)&=&N(c_{0}Ric+\Sigma^{\infty}_{\lambda_{i}=\mu_{i}}c_{i}h_{i}+\Sigma^{\infty}_{\lambda_{j}\not=\mu_{j}}c_{j}h_{j})\nonumber\\
	&=&N(\Sigma^{\infty}_{\lambda_{i}=\mu_{i}}c_{i}h_{i})\nonumber\\
	&=&\Sigma^{\infty}_{\lambda_{i}=\mu_{i}}c_{i}N(h_{i})\nonumber\\
	&=&\Sigma^{\infty}_{\lambda_{i}=\mu_{i}}c_{i}(\frac{\lambda_{i}}{2} +\frac{1}{2\tau})h_{i}\nonumber
\end{eqnarray}
And for the second variation of the  $\nu$-entropy in the direction $h$ we have
\begin{eqnarray}
	\nu^{''}_{g}(h)&=&\int_{M}<N(h),h>dm\nonumber\\ 
	&=&\int_{M}<\Sigma^{\infty}_{\lambda_{i}=\mu_{i}}c_{i}(\frac{\lambda_{i}}{2} +\frac{1}{2\tau})h_{i},\displaystyle c_{0}Ric+\Sigma^{\infty}_{\lambda_{r}=\mu_{r}}c_{r}h_{r}+\Sigma^{\infty}_{\lambda_{s}\not=\mu_{s}}c_{s}h_{s}>dm\nonumber\\
	&=&\int_{M}<\Sigma^{\infty}_{\lambda_{i}=\mu_{i}}c_{i}(\frac{\lambda_{i}}{2} +\frac{1}{2\tau})h_{i},\Sigma^{\infty}_{\lambda_{r}=\mu_{r}}c_{r}h_{r}>dm\nonumber
\end{eqnarray}
In the last equation we have used the fact that $N$ is a self-adjoint operator. Now because our basis is orthonormal therefore
\begin{eqnarray}
	\nu^{''}_{g}(h)
	&=&\int_{M}\Sigma^{\infty}_{\lambda_{i}=\mu_{i}}{c_{i}}^{2}(\frac{\lambda_{i}}{2} +\frac{1}{2\tau})|h_{i}|^{2}dm\nonumber\\
	&=&\Sigma^{\infty}_{\lambda_{i}=\mu_{i}}{c_{i}}^2(\frac{\lambda_{i}}{2} +\frac{1}{2\tau})\nonumber
\end{eqnarray}
But we assume that except  the zero eigenvalue  with the $Ric$ as only generator of eigenspace,  $\lambda_{i}\leqslant-\frac{1}{\tau}$ therefore
$\lambda_{i}+\frac{1}{2\tau}\leqslant0$ so $\nu^{''}_{g}(h)\leqslant0$ i.e soliton is stable. 
\end{proof}

Unfortunately for case $\lambda\in(-\frac{1}{\tau},-\frac{1}{2\tau})$ we have not any  information about stability of soliton.
   Here we give some information about the eigenvalues of the stability operator $N$.

   \begin{theorem}
   	 For a closed oraiantable $GRS^+$ $(M^{n}, g,f,\tau)$ all  eigentensors of the stability operator except eigentensors of the zero eigenvalue are eigentensors of the weighted Lichnerowicz Laplacian $\Delta_{f,L}$.
   	 
   	\end{theorem}
   \begin{proof}   	  
   	  Suppose that  $h_{i}\in C^{\infty}(S^2(T^*(M)))$ is an eigentensor of $\Delta_{f,L}$ with eigenvalue $\lambda_{i}\not=0$.  First we have $$div_{f}{h_{i}}=0$$    
   	   We have 
   	   \begin{eqnarray}
   	   	\int_{M}|div_{f}{N(h_{i})}|^{2}dm&=&\int_{M}<N(h_{i}),div^\dag_{f}div_{f}N(h_{i})>dm\nonumber\\&=&\int_{M}<h_{i},N(div^\dag_{f}div_{f}N(h_{i}))>dm\nonumber\\
   	   	&=&0\nonumber
   	   \end{eqnarray}	
       therefore  $div_{f}{N(h_{i})}=0$. Indeed for every  $h\in C^{\infty}(S^2(T^*(M)))$ such that $\Delta_{f,L}h\not=0$,   $div_{f}N(h)=0$.	 Secondly since $\lambda_{i}\not=0$ so 
   	 \begin{eqnarray}
   	 	\displaystyle{\int_{M}}<Ric,h_{i}>dm&=&\int_{M}<Ric,{\frac{\lambda_{i}}{\lambda_{i}}}  \  h_{i}>dm\nonumber\\
   	 	&=&\int_{M}<Ric,{\frac{1}{\lambda_{i}}}  N(h_{i})>dm\nonumber\\
   	 	&=&\int_{M}<N(Ric),{\frac{1}{\lambda_{i}}}h_{i}>dm\nonumber\\
   	 	&=&0\nonumber
   	 \end{eqnarray}
        Therefore for the stability operator of $h_{i}$ we have
    \begin{eqnarray}
     N(h_{i})=\frac{1}{2}\Delta_{f,L}h_{i}+\frac{1}{2\tau}h_{i}=\lambda_{i}h_{i}\nonumber    
    \end{eqnarray}
 And finally $$\Delta_{f,L}h_{i}=2(\lambda_{i}-\frac{1}{2\tau})h_{i}$$
   	 
   \end{proof}

 Hall and Morphy  in  \cite{hall} proved that any compact K\"ahler  $GRS^+$ with  $dim H^{1,1}(M)\geqslant2$ is linear unstable. For this purpose  they define the map
\begin{eqnarray}
	S:{\Omega ^{(1,1)}}(M) \to {C^\infty }({S^2}({T^*}M))\nonumber 
\end{eqnarray}
such that for twisted harmonic 1-1 form $(\omega,\Delta_{f,H}\omega=0)$, we have $div_f S(\omega)=0$  and  $\Delta_{f,L}S(\omega)=S(\Delta_{f,H}\omega)=0$. Then they consider a linear combination of two  twisted harmonic forms $\omega=a.\omega_{1}+b.\omega_{2}$  such that image of this linear combination is perpendicular to Ricci tensor $(\int_{M} { < S(\omega),Ric > dm} =0)$.
Now  the image of this linear combination under map $S$ is an unstability direction. Here we extend their result.
\newtheorem{hall}{Theorem}[section]
\begin{theorem}
	 Any compact orientable K\"ahler  $GRS^+$ $(M^{n}, g,f,\tau)$ with  $dim H^{1,1}(M)=1$ is linear unstable unless for twisted harmonic form $(\omega,\Delta_{f,H}\omega=0)$ and complex structure $J$, we have $S(\omega)=\omega(J,-)=\lambda Ric$ i.e $S(\omega)_{ij}=g^{rs}\omega_{ir}J_{sj}=\lambda R_{ij}$,$\lambda\in\mathbb{R},\lambda\not=0$
	\end{theorem}
\begin{proof}

	 Similar to case $\Delta_{f,L}h=0$ in Theorem (\ref{nec}) and because  $\Delta_{f,L}S(\omega)=0$ and $div_{f}S(\omega)=0$ (see \cite{hall}Lemma(4.1),Proposition(4.2),Lemma(4.3)) for $h=S(\omega)$ we conclude that
	 
	 \begin{eqnarray}
	 	{\nu _g}^{''}(h)&\geqslant&\int_{M} \frac{1}{2\tau}|h|^{2}dm  -\frac{({\int_{M} { < Ric,h > dm} })^2}{{2\tau\int_{M}  {|Ric|^{2}dm} }} \nonumber\\
	 	&\geqslant&\Big(\int_{M} \frac{1}{2\tau}|h|^{2}dm\int_{M} 2\tau  {|Ric|^{2}dm}  -({\int_{M} { < Ric,h > dm} })^2\Big)\frac{1}{{{2\tau\int_{M}  {|Ric|^{2}dm} }}} \nonumber\\
	 	&\geqslant&\Big(\int_{M} |h|^{2}dm\int_{M}  {|Ric|^{2}dm}  -({\int_{M} { < Ric,h > dm} })^2\Big)\frac{1}{{{2\tau\int_{M}  {|Ric|^{2}dm} }}} \nonumber
	 \end{eqnarray}
	 But according to Cauchy–Schwarz inequality we know that the last term is positive and is zero if and only if $h=\mu Ric , \mu\in \mathbb{R},\mu\not=0$ but we assume that  $h\not=\mu Ric , \mu\in \mathbb{R},\mu\not=0$, therefore $\nu^{''}_{g}(h)>0$ and soliton is unstable.
	 \end{proof}


\begin{thebibliography}{5}
\bibitem{berger}
Berger, M. and Ebin, D., Some decompositions of the space of symmetric tensors on a Riemannian manifold, J. Differential Geom., 3(1969), 379-392.
\bibitem{cao}
Cao, H.-D.,  Recent progress on Ricci solitons, Recent advances in geometric analysis,
 Adv. Lect. Math. (ALM), 11." Int. Press, Somerville, MA (2010), 1–38.
\bibitem{cao zhu}
Cao, H.-D and ZHU, M., On second variation of Perelman’s Ricci shrinker entropy, Math. Ann, 353 (2012), no. 3, 747–763.

\bibitem{hail}
Cao, H.-D., Hamilton,R., Ilmanen, T., Gaussian densities and stability for some Ricci solitons, arXiv preprint math/0404165 (2004).
\bibitem{sym}
Cao, H.-D, ; HE, C., Linear stability of Perelman's  $\nu$-entropy on symmetric spaces of compact type, arXiv preprint /1304.2697, (2013).
\bibitem{chow}
 Chow, B., Chu, S.-.C, Glikenstein, D.,  Guenther, C., Isenberg, J., Ivey, T., Knopf, D., Lu, P., Luo, F., Ni, L., The Ricci flow:
techniques and applications. Part I: Geometric aspects,  Mathematical Surveys and Monographs 135. Providence, RI: American Mathematical Society (AMS). xxiii, , 2007, 536.
\bibitem{alex}
Deruelle, A., Stability of non compact steady and expanding gradient Ricci solitons, Calculus of Variations and Partial Differential Equations volume 54,  (2015), 2367–2405.
\bibitem{futaki}
 Futaki, A. and Sano, Y., Lower diameter bounds for compact shrinking Ricci solitons, arXiv preprint/1007.1759.
\bibitem{hall}
Hall, S. and Murphy, T., On the linear stability of Ka¨ahler-Ricci solitons, Proc. Amer. Math.Soc. 139 (2011), 3327–3337.
\bibitem{formation}
Hamilton, R. S., The formation of singularities in the Ricci flow, Surveys in Differential Geometry (Cambridge, MA, 1993), 2, , International Press,Combridge, MA, 1995, 7-136.
\bibitem{koiso}
Koiso, N., Rigidity and stability of Einstein metrics–the case of compact symmetric spaces, Osaka J. Math. 17(1980), no. 1, 51–73.
\bibitem{kroncke}
Kr\"oncke, K., Stability and instability of Ricci solitons,  Calc. Var. PDE. 53, (2015), 265-287.
\bibitem{thesis}
Kr\"oncke, K., Stability of Einstein manifolds, Universit\"at Potsdam, PhD thesis, 2013.
\bibitem{lich}
 Lichnerowicz, A., Propagateurs et commutateurs en relativit\'e  g\'en\'erale, Inst.Hautes Etudes Sci. Publ. Math, 10 (1961), 293-344.
\bibitem{lott}
Lott, J., Some geometric properties of the Bakry-Emery-Ricci tensor, ´ Comment.Math. Helv. 78(2003) 865-883.
\bibitem{curv}
Munteanu, O. and Wang M.-T., The curvature of gradient Ricci solitons, Math. Res. Lett.,18(6), (2011), 1051–1069.

\bibitem{classif}
 Ni, L. and Wallach, N., On a classification of the gradient shrinking solitons, Math. Res. Lett. 15 (2008), 941-955.
 \bibitem{per}
 PERELMAN, G., The entropy formula for the Ricci flow and its geometric applications, arXiv preprint/ math/0211159, (2002).
 \bibitem{Demys}
  Petersen,P., Demystifying the Weitzenb\"ock curvature operator, 2011, http://www.math. ucla.edu/~petersen/BLWformulas.pdf.

 \bibitem{wyl}
  Petersen, P. and Wylie, P., On the classification of gradient Ricci solitons, Geom. Topol. 14 (2010), 2277-2300.
\bibitem{topping}
 Topping, T., Lectures on the Ricci flow, London Mathematical Society Lecture Note Series,vol. 325, Cambridge University Press, Cambridge, 2006.
\bibitem{step}
 Rovenski, V., Stepanov, S. and  Tsyganok, I.,   On the geometry in the large of Lichnerowicz type Laplacians and its applications,  Balkan Journal of Geometry and Its Applications, Vol.25, No.2, 2020,  76-93.




 \end{thebibliography}
\end{document}